\documentclass[a4paper,
  headings=standardclasses
]{scrartcl}
\usepackage[a4paper,margin=2.0cm]{geometry}

\usepackage{mathtools}
\usepackage{amssymb}
\usepackage{amsthm}
\usepackage{algorithm}
\usepackage{caption} 		
\usepackage{subcaption}

\usepackage{tikz}
\usetikzlibrary{decorations.markings}

\usepackage[hidelinks]{hyperref}
\usepackage{cleveref}

\usepackage{xcolor}
\hypersetup{
    colorlinks,
    linkcolor={red!50!black},
    citecolor={blue!50!black},
    urlcolor={blue!80!black}
}

\usepackage[status=draft,layout=footnote]{fixme}
\fxsetup{theme=color}

\newcommand{\NN}{\mathbb{N}}

\newcommand{\ZZ}{\mathbb{Z}}

\newcommand{\cP}{\mathcal{P}}

\DeclareMathOperator{\Ker}{Ker}

\newcommand{\BS}{BS(2,3)}

\newtheorem{theorem}{Theorem}
\newtheorem{corollary}[theorem]{Corollary}
\newtheorem{lemma}[theorem]{Lemma}
\newtheorem{proposition}[theorem]{Proposition}
\theoremstyle{remark}

\newtheorem{remark}[theorem]{Remark}
\theoremstyle{definition}
\newtheorem{definition}[theorem]{Definition}

\begin{document}

	\title{A closer look at the non-Hopfianness of $\BS$}
	
	\author{Tom Kaiser}
	
	\maketitle
	
	\begin{abstract}
		The Baumslag-Solitar group $\BS$, is a so-called non-Hopfian group, meaning that it has an epimorphism $\phi$ onto itself, that is not injective. In particular this is equivalent to saying that $\BS$ has a non-trivial quotient that is isomorphic to itself. As a consequence the Cayley graph of $\BS$ has a quotient that is isomorphic to itself up to change of generators. We describe this quotient on the graph-level and take a closer look at the most common epimorphism $\phi$. We show its kernel is a free group of infinite rank with an explicit set of generators. Finally we show how $\phi$ appears as a morphism on fundamental groups induced by some continuous map. This point of view was communicated to the author by Gilbert Levitt.
	\end{abstract}

	
	\tikzset{->-/.style={decoration={
				markings,
				mark=at position .5 with {\arrow{>}}},postaction={decorate}}}

	\section{Introduction}

	Baumslag-Solitar groups are two-generator one-relator groups given by the presentation
	\[ BS(n,m) =  \langle a,b \;\lvert \; ba^n =  a^m b\rangle,\] 
	for $n,m\in\ZZ$. They were introduced in 1962 by Baumslag and Solitar\footnote{The group $BS(1,2)$ was already around at the time and appears in a 1951 paper by Higman \cite{higman}.} to describe the first examples of non-Hopfian finitely generated one-relator groups. A group is Hopfian whenever $G/N\cong G$ implies that $N=\{1\}$. The specific result \cite{BS} states that $BS(n,m)$ is Hopfian if and only if
	\begin{itemize}
		\item $n$ divides $m$ or vice versa; or
		\item $n$ and $m$ have the same prime divisors.
	\end{itemize}
	There were some errors in the original proofs. This issue was later resolved by Collins and Levin in \cite{collinslevin}. It is also possible to determine which of these groups are residually finite \cite{meskin}. This is the case when $\lvert m \rvert = \lvert n \rvert$, or $\lvert n\rvert =1$ or $\lvert m \rvert =1$. A last interesting remark is that, whenever $\lvert n\rvert, \lvert m\rvert \ne 1$, these groups are HNN extenstions of $\mathbb{Z}$ with respect to the automorphism $\mu: n\mathbb{Z}\rightarrow m\mathbb{Z}$, sending $n$ to $m$.\newline

	We only focus on the Baumslag-Solitar group $BS(2,3)= \langle a,b \;\lvert \; ba^2 =  a^3b\rangle$ and the following homomorphism which is surjective, but not injective: 
	\[   \phi :  BS(2,3) \rightarrow BS(2,3) :  
	\begin{cases}
	a \mapsto a^2\\
	b\mapsto b \\
	\end{cases}
	. \]
	The element $[a^b,a]$ is non-trivial, and inside the kernel of $\phi$. As mentioned before we notice that $\BS/ \Ker(\phi)$ is isomorphic to $\BS$. And thus their Cayley graphs are isomorphic up to change of generators. We will first describe this quotient from a graph-viewpoint. Our other aim is to show that $\Ker(\phi)$ is a free group of infinite rank. The proof is geometric. We will give an explicit set of generators and show that their action on the Bass-Serre tree associated to $\BS$ is free. Finally we present how $\phi$ appears as a morphism on fundamental groups, induced by a continuous map.\newline
	
	It is useful to remind ourselves of Britton's lemma. We do not give its general version, but directly apply it to $\BS$. 
	
	\begin{lemma}[Britton's lemma]
	Let $\omega= a^{\alpha_0}\prod_{i=1}^{n} b^{\beta_i}a^{\alpha_i}$ in $\BS$ such that $\beta_i\ne 0$. If $\omega =1$, then either 
	\begin{itemize}
	\item $n=0$ and $\alpha_0=0$;
	\item or $n>0$  and $\exists i\in \{1\dots n-1\} $ such that
	\begin{enumerate}
	\item $\beta_i>0,  \beta_{i+1}<0$ and $\alpha_{i}\in 2^{\min\{\beta_i,\lvert\beta_{i+1}\rvert\}} \mathbb{Z}$;
	\item $\beta_i<0,  \beta_{i+1}>0$ and $\alpha_{i}\in 3^{\min\{\lvert\beta_i\rvert,\beta_{i+1}\}} \mathbb{Z}$.
	\end{enumerate}
	\end{itemize}
	\end{lemma}		
	
	\section*{Acknowledgements}
	The author would like to thank Gilbert Levitt for some very interesting comments. In particular the point of view presented in Section \ref{levittpres} is due to Levitt, as is Figure \ref{Levittpic}.

	\section{An interpretation of non-Hopfianness and a quotient of graphs}
		
	We first calculate the kernel of $\phi$, then remind the reader what the Cayley graph of $\BS$ looks like, in order to finally describe its quotient induced by the morphism $\phi$. 
	
	\begin{definition}
	Given a word $\omega= a^{\alpha_0} \prod_{i=1}^{n} a^{\alpha_i}b^{\beta_i}$ in $\{a,b\}^*$, define $\rho(\omega)=\sum_{i=1}^{n} \beta_i$ and $\rho_{a}(\omega) = \sum_{i=1}^{n} \lvert\beta_i\rvert$. 
	\end{definition}
	
	Note that, as a consequence of Britton's lemma, we see that if $\omega$ represents the trivial element, then $\rho(\omega)=0$.
	
	\begin{lemma}
		$\Ker(\phi)$ is normally generated by $\{ [a^b,a], [a^b, a^{-1}], [a^b,a^2], [a^b, a^{-2}]  \}$.
	\end{lemma}
	\begin{proof}
	
	Consider a word $\omega$ in the generators of $BS(2,3)$. If $\omega\in \Ker(\phi)$, then $
	\rho(\omega)=0$, since $\rho(\phi(\omega))=0$ and $\phi$ does not affect $b$. Given that a proof is a sequence of trivialities, we can
	check the following:
	
	\begin{itemize}
		\item If $\rho_a(\omega)=2$, then an easy calculation shows that if $\phi(\omega)=1$, then  $\omega$ represents the trivial element.
		\item If $\rho_a(\omega)=4$, one can show that $\omega$ represents one of the generators given above (up to 
		cyclic permutation of the letters).
		\item Consider the group element $b^\beta a^\alpha b^{\beta'}$, where $\beta,\beta'>0$. By Bachet-Bezout, there exist $\lambda,\mu\in \ZZ$ such that $\lambda 2^\beta + \mu 3^{\beta'} =1$. Hence $b^\beta a^\alpha b^{\beta'} = b^\beta a^{\alpha\lambda 2^\beta + \alpha\mu 3^{\beta'}} b^{\beta'} $, which equals $a^{\alpha\lambda 3^{\beta} } b^{\beta+\beta'} a^{\alpha\mu 2^{\beta'} }$.
		\item By repeatedly applying the previous item, any word $\omega$ (up to cyclic permutation) can be replaced by a word 
		$\omega'$ that represents the same element in $BS(2,3)$ and is of the form $a^{\alpha_1}b^{\beta_1}a^{\alpha_1}b^{\beta_2}\dots a^{\alpha_n}b^{\beta_n} $, where none of the exponents are zero and exponents of subsequent $b$'s have opposite sign.
		\item We specialise even more, in the sense that we minimise the exponents $\beta_i$. This means that
		 $ba^{2\alpha} b^{-1}$ can be replaced by $a^{3\alpha}$ and $b^{-1} a^{3\alpha}b$ by $a^{2\alpha}$.
		\item So suppose $\omega$ is of the prescribed minimal form, such that $\rho_a(\omega)>4$ and is in 
		$\Ker(\phi)$, then the image is  $a^{2\alpha_1}b^{\beta_1}a^{2\alpha_2}b^{\beta_2}\dots a^{2\alpha_n}b^{\beta_n} $.  We apply Britton's lemma. Suppose there is some $i\in\{ 1\dots n\}$ such that $\beta_i>0,  \beta_{i+1}<0$ and $2\alpha_{i}\in 2^{\min\{\beta_i,\lvert\beta_{i+1}\rvert\}} \mathbb{Z}$. Without loss of generality we suppose $\beta_i$ is the minimum. Then $\alpha_i= 2^{\beta_i-1}\alpha_i'$, and thus $b^{\beta_i} a^{\alpha_i}b^{\beta_{i+1}}= b a^{3^{\beta_i-1}\alpha_i'} b^{\beta_{i+1}+\beta_i-1}$. Note that this contradicts minimality of the word $\omega$ if $\beta_i>1$. The reader can check that a similar contradiction is found for the second case of Britton's lemma.		
		\item Because of the previous point, there is some $b$ whose exponent is either $1$ or $-1$. Up to taking inverses and cyclicly permuting the word, we may suppose it is $-1$ and it appears on the second position, i.e. $\omega$ can be chosen of the form $b^{\beta_1}a^{\alpha_2}b^{-1}a^{\alpha_3} b^{\beta_3}\dots$
			\item Adding minimality to the argument, this can be rewritten as
			\[b^{\beta_1 - 1} a^{3(\frac{\alpha_2-1}{2}) + \alpha_3 -t} b a b^{-1} a^{t}  b a^{-1}b^{-1 } a^{- t}  a^{t} b a b^{\beta_3 -1} \dots\]
			where $t$ is either $1$ or $2$, such that $\alpha_3-t\in 3\ZZ$. Note that the first part of our rewritten word is exactly  a generator described in the statement of the lemma. Suppose $\omega'$ is the remaining word, after deleting the generator. One notices that $\rho_a(\omega')=\rho_a(\omega)-2$. Hence an induction argument based on the value of $\rho_a$ finishes the proof.
		\end{itemize}

	\end{proof}

	\begin{corollary}
		$\Ker(\phi)$ is normally generated by $[a^b,a]$.
	\end{corollary}
	\begin{proof}
		Since all other generators can by normally generated by $[a^b,a]$:
		\begin{itemize}
			\item $[a^b,a^{-1}] = a^{-1}[a^b,a]^{-1}a$,
			\item $[a^b,a^{2}] = [a^b,a] a [a^b,a]a^{-1}$, and
			\item $[a^b,a^{-2}] = a^{-1}[a^b,a]^{-1} a^{-1} [a^b,a]^{-1} a^2$.
		\end{itemize}
	\end{proof}

	A description of $BS(1,2)$ and its Cayley graph is a pretty standard introduction to the world of Baumslag-Solitar groups (f.e. \cite{meier2008groups}). For self-containment purposes, we quickly describe the Cayley graph of $BS(2,3)$, which follows a slight variation of the classical example $BS(1,2)$. The standard building block is given in Figure \ref{buildingblock}, where red edges are labeled by $a$ and blue ones by $b$. This is the cycle induced by the single relator. These building blocks fit nicely together into an upper-half plane given in Figure \ref{plane}, based at a copy of $\ZZ$ generated by $a$. We may now construct the Cayley graph. Consider the base line generated by $a$, then we can insert three planes here such that upward arrows do not overlap. One can see this as attaching a plane based at $1$, $a$ and $a^2$. In a similar way we construct lower half planes (one can see this by interpreting Figure \ref{plane} with the baseline on top). Two of these can be attached at $\ZZ$ such that blue arrows do not overlap. Locally this looks like Figure \ref{side3dview}. In general consistenly applying this construction to each left coset of $\langle a \rangle$, which all appear as a Cayley graphs of $\ZZ$, gives us the Cayley graph of $\BS$. Now project such that each left coset of $\langle a \rangle$ becomes a point, we obtain a tree of degree five, as can be seen in Figure \ref{sideview}. This is the Bass-Serre tree associated to the HNN extension $\BS = \text{HNN}(\ZZ,2\ZZ\sim 3\ZZ)$. Choosing a specific vertex of this tree, one sees that three edges are directed upward and two downward (the halfplanes attached upward and downward respectively). This `side view' will be important for our visual interpretation later. We first define the following sets.
	
\begin{minipage}{\linewidth}
		\centering
		\begin{minipage}{0.3\linewidth}
			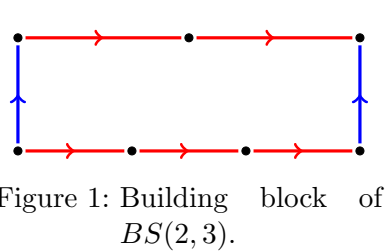
\begin{figure}[H]
				\centering
				\begin{tikzpicture}
				\filldraw (0,0) circle (1.5pt);
				\draw[->-, very thick,red] (0.1,0) -- (1.4,0);
				\filldraw (1.5,0) circle (1.5pt);
				\draw[->-, very thick,red] (1.6,0) -- (2.9,0);
				\filldraw (3,0) circle (1.5pt);
				\draw[->-, very thick,red] (3.1,0) -- (4.4,0);
				\filldraw (4.5,0) circle (1.5pt);
						
				\draw[->-, very thick,blue] (0,0.1) -- (0,1.4);
				\filldraw (0,1.5) circle (1.5pt);
				\draw[->-, very thick,blue] (4.5,0.1) -- (4.5,1.4);
				\filldraw (4.5,1.5) circle (1.5pt);
							
				\draw[->-, very thick,red] (0.1,1.5) -- (2.15,1.5);
				\filldraw (2.25,1.5) circle (1.5pt);
				\draw[->-, very thick,red] (2.35,1.5) -- (4.4,1.5);
				\end{tikzpicture}
				\caption{Building block of $\BS$.}
				\label{buildingblock}
			\end{figure}
		\end{minipage}
		\begin{minipage}{0.64\linewidth}
			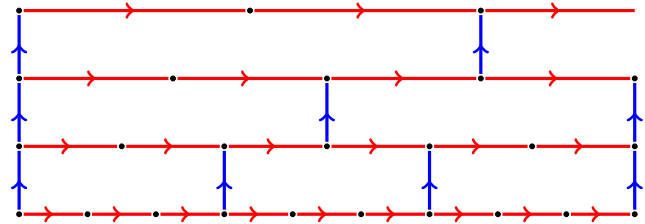
\begin{figure}[H]
				\centering
				\begin{tikzpicture}[scale=0.6]
				\filldraw (0,0) circle (1.5pt);
				\draw[->-, very thick,red] (0.1,0) -- (1.4,0);
				\filldraw (1.5,0) circle (1.5pt);
				\draw[->-, very thick,red] (1.6,0) -- (2.9,0);
				\filldraw (3,0) circle (1.5pt);
				\draw[->-, very thick,red] (3.1,0) -- (4.4,0);
				\filldraw (4.5,0) circle (1.5pt);
				\draw[->-, very thick,red] (4.6,0) -- (5.9,0);
				\filldraw (6,0) circle (1.5pt);
				\draw[->-, very thick,red] (6.1,0) -- (7.4,0);
				\filldraw (7.5,0) circle (1.5pt);
				\draw[->-, very thick,red] (7.6,0) -- (8.9,0);
				\filldraw (9,0) circle (1.5pt);
				\draw[->-, very thick,red] (9.1,0) -- (10.4,0);
				\filldraw (10.5,0) circle (1.5pt);
				\draw[->-, very thick,red] (10.6,0) -- (11.9,0);
				\filldraw (12,0) circle (1.5pt);
				\draw[->-, very thick,red] (12.1,0) -- (13.4,0);
				\filldraw (13.5,0) circle (1.5pt);

				\draw[->-, very thick,blue] (0,0.1) -- (0,1.4);
				\filldraw (0,1.5) circle (1.5pt);
				\draw[->-, very thick,blue] (4.5,0.1) -- (4.5,1.4);
				\filldraw (4.5,1.5) circle (1.5pt);
				\draw[->-, very thick,blue] (9,0.1) -- (9,1.4);
				\filldraw (9,1.5) circle (1.5pt);
				\draw[->-, very thick,blue] (13.5,0.1) -- (13.5,1.4);
				\filldraw (13.5,1.5) circle (1.5pt);

				\draw[->-, very thick,red] (0.1,1.5) -- (2.15,1.5);
				\filldraw (2.25,1.5) circle (1.5pt);
				\draw[->-, very thick,red] (2.35,1.5) -- (4.4,1.5);
				
				\draw[->-, very thick,red] (4.6,1.5) -- (6.65,1.5);
				\filldraw (6.75,1.5) circle (1.5pt);
				\draw[->-, very thick,red] (6.85,1.5) -- (8.9,1.5);
				
				\draw[->-, very thick,red] (9.1,1.5) -- (11.15,1.5);
				\filldraw (11.25,1.5) circle (1.5pt);
				\draw[->-, very thick,red] (11.35,1.5) -- (13.4,1.5);

				\draw[->-, very thick,blue] (0,1.6) -- (0,2.9);
				\filldraw (0,3) circle (1.5pt);
				
				\draw[->-, very thick,blue] (6.75,1.6) -- (6.75,2.9);
				\filldraw (6.75,3) circle (1.5pt);
				
				\draw[->-, very thick,blue] (13.5,1.6) -- (13.5,2.9);
				\filldraw (13.5,3) circle (1.5pt);

				\draw[->-, very thick,red] (0.1,3) -- (3.275,3);
				\filldraw (3.375,3) circle (1.5pt);
				\draw[->-, very thick,red] (3.475,3) -- (6.65,3);
				
				\draw[->-, very thick,red] (6.85,3) -- (10.025,3);
				\filldraw (10.125,3) circle (1.5pt);
				\draw[->-, very thick,red] (10.225,3) -- (13.4,3);

				\draw[->-, very thick,blue] (0,3.1) -- (0,4.4);
				\filldraw (0,4.5) circle (1.5pt);
				
				\draw[->-, very thick,blue] (10.125,3.1) -- (10.125,4.4);
				\filldraw (10.125,4.5) circle (1.5pt);

				\draw[->-, very thick,red] (0.1,4.5) -- (4.9625,4.5);
				\filldraw (5.0625,4.5) circle (1.5pt);
				\draw[->-, very thick,red] (5.1625,4.5) -- (10.025,4.5);
				
				\draw[->-, very thick,red] (10.225,4.5) -- (13.5,4.5);
				\end{tikzpicture}
				\caption{Construction of a plane in $\BS$.}
				\label{plane}
			\end{figure}
		\end{minipage}
	\end{minipage}
	\newline

	\begin{minipage}{\linewidth}
			\centering
			\begin{minipage}{0.25\linewidth}
				\begin{figure}[H]
					\centering
					  \includegraphics[width=\linewidth]{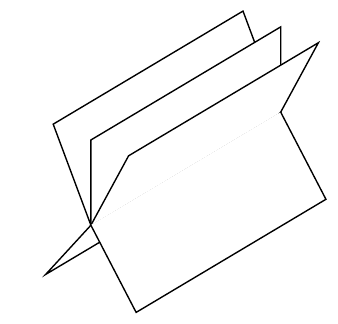}
					\caption{Side 3D view of $\BS$.}
					\label{side3dview}
				\end{figure}
			\end{minipage}
			\begin{minipage}{0.64\linewidth}
							\begin{figure}[H]
								\centering
								\includegraphics[width=\linewidth]{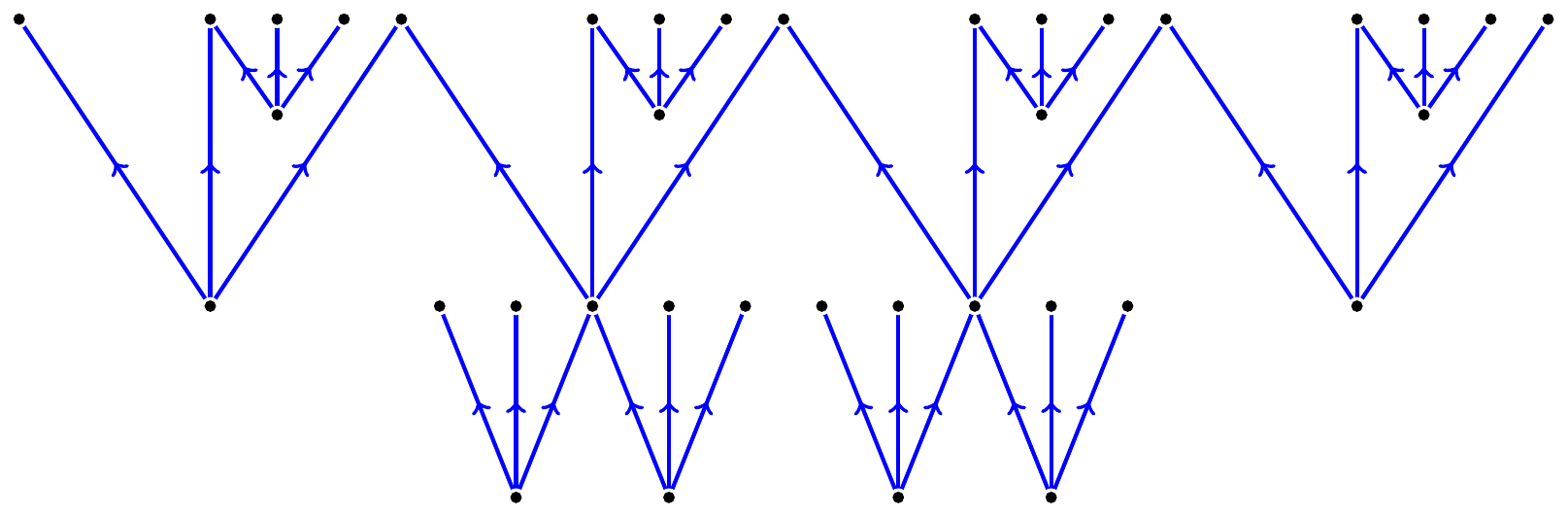}
								\caption{Projection of cosets of $\langle a\rangle$.}
								\label{sideview}
							\end{figure}
						\end{minipage}
		\end{minipage}
		\newline

	We define the height of an element to be its image by the map 
	\[  \alpha: \BS \rightarrow \ZZ: \left\{  \begin{array}{lll}
	a &\mapsto& 0\\
	b &\mapsto& 1
	\end{array}\right.,  \]		
	and we define what it means for two cosets to be neighbours.
	\begin{definition}
		Two left cosets $\lambda$ and $\mu$ of $\langle a \rangle$ are same height neighbours if for $l\in \lambda$ either $l a^b$, $l a a^b$ or $l a^2 a^b$ are in $\mu$. Notation: $\lambda \sim \mu$. 
	\end{definition}
	If we consider the graph defined by those cosets as vertices and edges between two cosets that are same height neighbours, then we obtain a forest (read: disjoint union) of $3$-valent trees. Note that such a tree has a unique $2$-colouring. This colouring can be seen in Figure \ref{coloringH} and thus also the visual interpretation of the same height neighbour relation becomes clear. Presented is a limited sideview of $\BS$, where the cosets are coloured by the $2$-colouring. Two coloured points (representing left cosets of $\langle a \rangle$) that are two edges apart, are neighbours. From the visual it is clear that a point has exactly three neighbours. Now we can define the following sets:
	\begin{definition}
		Given $\lambda$ a left coset of $\langle a \rangle$, then 
		\begin{itemize}
			\item $H_\lambda$ is the connected component of $\lambda$ of the forest induced by the same height neigbour relation;
		\end{itemize}
		Since $H_\lambda$ is a $3$-valent tree it has a unique $2$-colouring.
		\begin{itemize}
			\item  $H_\lambda^+$ are the cosets in $H_\lambda$, that have the same colour as $\lambda$;
			\item $H_\lambda^-:= H_\lambda\backslash H_\lambda^+$.
		\end{itemize}
	\end{definition}
		Note that if $\mu \in H_\lambda^-$, then clearly $H_\lambda^- =  H_\mu^+$.

\begin{figure}
  \includegraphics[width=\linewidth]{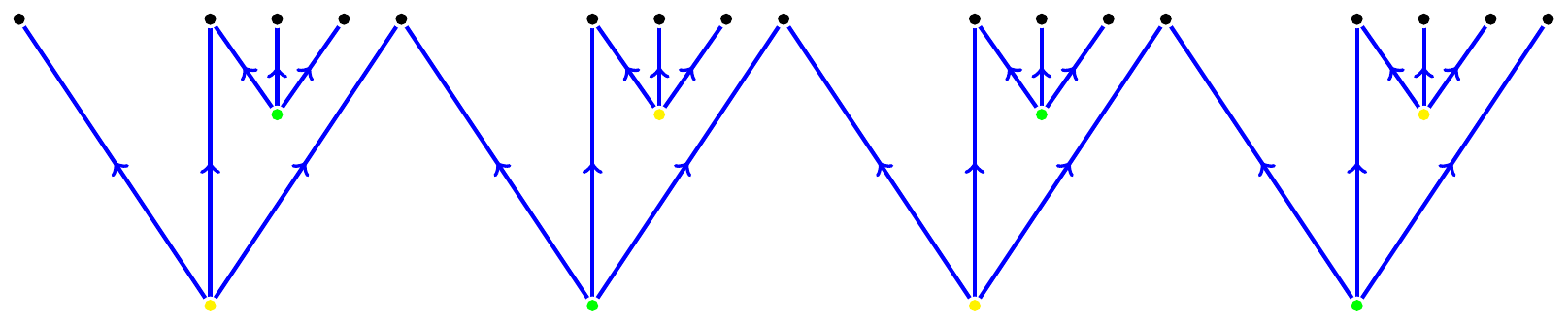}
  \caption{Coloring of $H_\lambda$.}
  \label{coloringH}
\end{figure}

	We now have the tools to show what happens when we quotient by $\Ker(\phi)$. Write \linebreak $G=  BS(2,3)/ \Ker(\phi)$, which we will not consider as an abstract group, but rather as the specific manifestation of $\BS$ as a quotient of itself. Note that $\phi(a)=a^2$. This means that $\overline{a}$, in $G$, will behave like $a^2$ in $BS(2,3)$. In a sense we would like to take the root of $\overline{a}$ (note that the reason for this awkward phrasing comes from the visual interpretation we will see soon). This is why we look for a pre-image of $a$, one can take for example the commutator $[b,a]= bab^{-1}a^{-1}$. So the specific group element $\overline{[b,a]}$ in  $G$ will behave as the abstract group element $a$ in $BS(2,3)$. For brevity we write $\tilde{a}:= \overline{[b,a]}$.\newline

	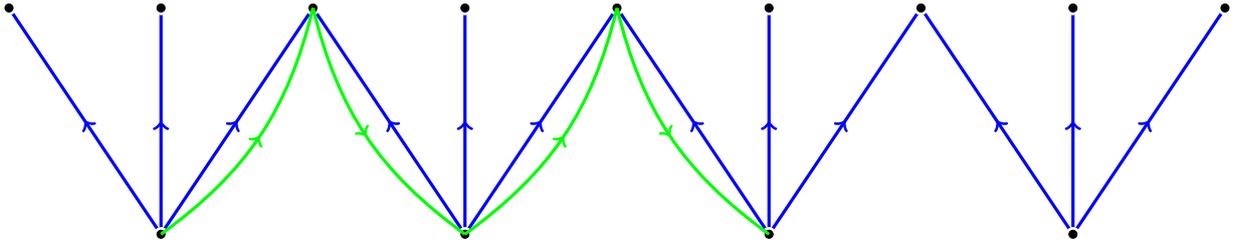
\begin{figure}
		\centering
	\begin{tikzpicture}
	\foreach \x in {0,...,3}
	{\draw[->-, very thick,blue] (4*\x,0) -- (4*\x,3);
		\draw[->-, very thick,blue] (4*\x,0) -- (4*\x +2,3);
		\draw[->-, very thick,blue] (4*\x,0) -- (4*\x -2,3);
		\filldraw[white] (4*\x,3) circle (2.5pt);
		\filldraw[white] (4*\x+2,3) circle (2.5pt);
		\filldraw[white] (4*\x-2,3) circle (2.5pt);
		\filldraw (4*\x,3) circle (1.5pt);
		\filldraw (4*\x+2,3) circle (1.5pt);
		\filldraw (4*\x-2,3) circle (1.5pt);
		\filldraw[white] (4*\x,0) circle (2.5pt);
		\filldraw (4*\x,0) circle (1.5pt);
	}
	
	\draw[->-,green,very thick] (0,0) to[bend right=20] (2,3);
	\draw[->-,green,very thick] (2,3) to[bend right=20] (4,0);
	\draw[->-,green,very thick] (4,0) to[bend right=20] (6,3);
	\draw[->-,green,very thick] (6,3) to[bend right=20] (8,0);
	\end{tikzpicture}
	\caption{The identification of the element $[a^b,a]$.}
	\label{treething}
	\end{figure}

	We will construct the quotient graph by identifying points in the original graph, following a principle we will call `horizontal identification'. Let us see which effect, identifying two points that are at a `distance' $\omega=[a^b,a^{-1}]$ has. We follow the path (in green) of $\omega$ in Figure \ref{treething}. Suppose we start on the line $\langle a \rangle$, then we go up a level to $b\langle a \rangle $. Next we move over an edge labelled $a$, which means that when we take the edge $b^{-1}$, we do not take the same edge back (i.e. we follow the second green arrow). After this, moving by $a^{-1}$, we go up one of the two other planes. Then we go down again. Note that since we identify two  points of $\langle a \rangle$  and $bab^{-1}a^{-1}ba^{-1}b^{-1} \langle a \rangle$, in fact those lines are completely identified. By symmetry all cosets of $H_{\langle a \rangle}^+$, will become equivalent to $\langle a \rangle$ in $G$. Note that by symmetry also all lines in $H_\lambda^-$ are identified with one another. Let us consider the behaviour of $\tilde{a}$. Note that it is exactly equal to the the first four letters of $\omega$. This means that we are basically following the first two green arrows of $\omega$ in Figure \ref{treething}. Next we apply $\tilde{a}$ a second time, we arrive on the line $\omega\langle a \rangle$, which has been identified to $\langle a \rangle$. Specifically the point we arrive in is identified with $a$, since $\tilde{a}=\overline{a}$. One notices that viewing $G$ wrt the new generators $\{\overline{b},\tilde{a} \}$ will give us an intertwining of the lines $\langle a \rangle$ and $bab^{-1} \langle a \rangle$. With respect to the sets $H_{\langle a \rangle}^+$ and $H_{\langle a \rangle}^-$, which are each just one line in $G$, this means that the generator $\tilde{a}$ will alternate points of $H_{\langle a \rangle}^+$ and $H_{\langle a \rangle}^-$. The result can be seen in Figure \ref{horizontalidentification}. Here points on the lower line are obtained by alternating points of $\langle a\rangle$ and $[b,a^{-1}]\langle a \rangle$, where we start by $1$ and $[b,a]$ respectively. Out of this new `baseline' once again three planes open up. We describe one, since the other two are analogous by symmetry. Consider the plane based at $\overline{1}$ in Figure \ref{horizontalidentification}. We take one step back and look at Figure \ref{loweridentification}, before identification. Then we see that the intertwining respects the upper brown points, in the sense that if the first green point from the left is $1$, then the first yellow point from the left is $[b,a]$. This means that the fourth point is exactly $[b,a]a$. Which is underneath the second brown point. This means that, after identification, a position opens up in between consecutive brown points. This makes sense, since by symmetry of the cayley graph, also the line $H_{b\langle a \rangle}^+$, must intertwine with $H_{b\langle a \rangle}^-$. This fills up the gap. Note that points in $H_{b\langle a \rangle}^-$ are exactly the orange points in Figure \ref{horizontalidentification}.
	
\begin{figure}
  \includegraphics[width=\linewidth]{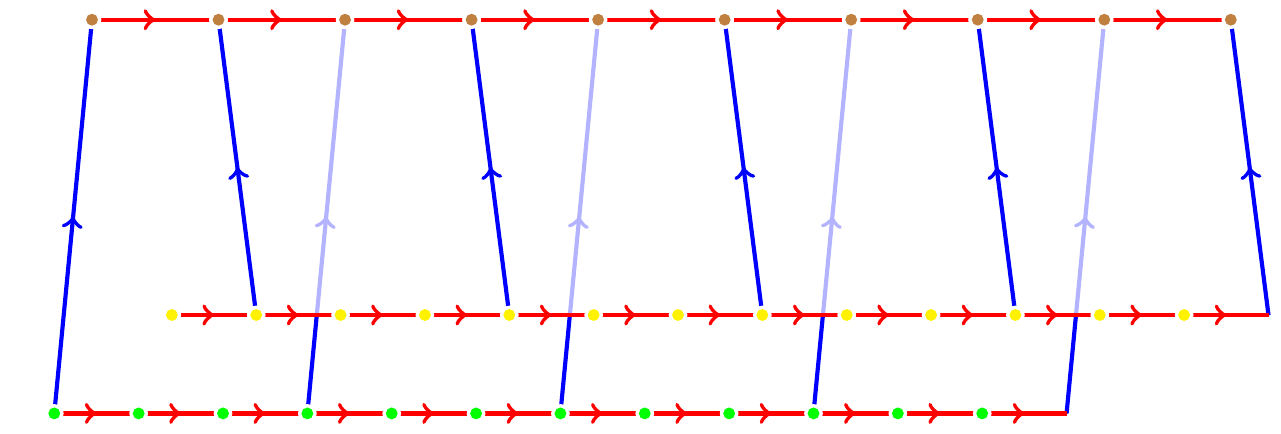}
  \caption{Two lines are interleaved if at a `distance' $bab^{-1}a^{-1}$.}
  \label{loweridentification}
\end{figure}

\begin{figure}
  \includegraphics[width=\linewidth]{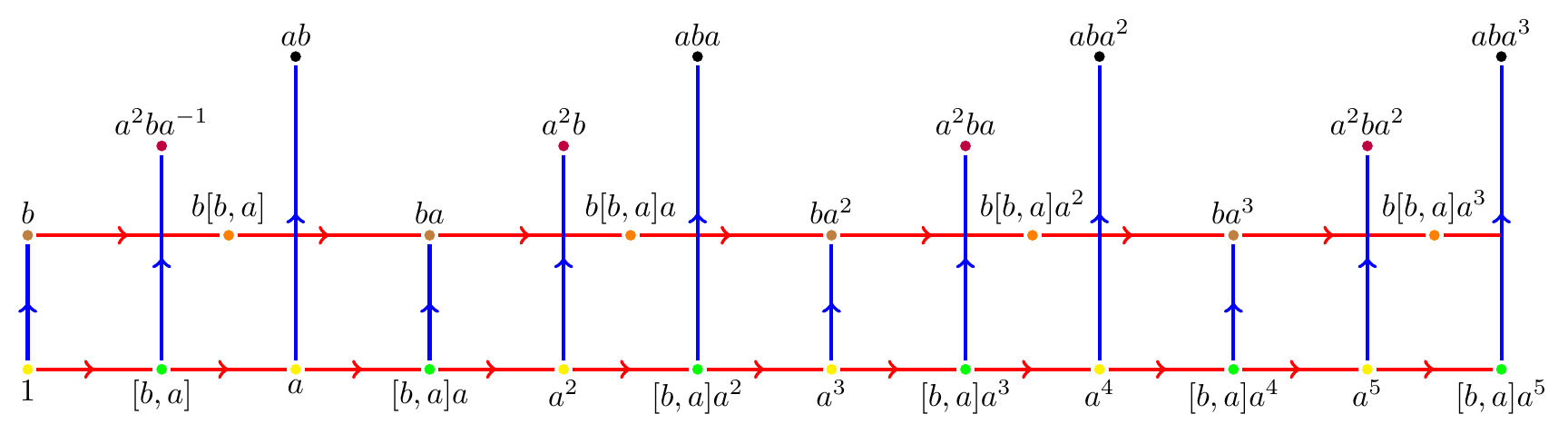}
  \caption{A part of the Cayley graph of $G$, where $\tilde{a}$ is now a generator. Choices of colours are consistent with previous  drawings.}
  \label{horizontalidentification}
\end{figure}

Note that since $\Ker(\phi)$ is normally generated by $[a^b,a]$, this is the only `real' identification that takes place. Other identifications are due to taking conjugates of this element. I.e. if two cosets of $\langle a\rangle$ are identified, of course the tree structure above and below it must become identified too.

\section{A free group of infinite rank}


Let us now prove that the kernel of $\phi$ is a free group of infinite rank. In order to do this in an accessible way, we first introduce some terminology. For illustrations of the concepts, also see Figure \ref{paths}.

\begin{definition}
	Let $\cP$ be the set of geodesic paths starting at $\langle a\rangle$ in the Bass-Serre tree in Figure \ref{sideview} (these are the left cosets of $\langle a \rangle$). Let $\{ a,b \}^*$ be all words in $a,b$ and their inverses, and $n,m\in\ZZ$. We define:
	\begin{enumerate}
		\item Let $\omega\in \{ a,b \}^*$, then the associated path in the Bass-Serre tree is the projection of the associated path in the Cayley graph of $\BS$. This path may contain backtracking.
		\item A good representative of a path $p\in\cP$ is a word $\omega$ such that $p$ is its associated path in the Bass-Serre tree. Equivalently: such that $\rho_a(\omega)$ is minimal and $\omega$ represents the left-coset at the end of $p$. 
		\item A path contains a tip if a good representative $\omega$ contains a subword of the form $ba^n b^{-1}$.
		\item A path contains a valley if a good representative $\omega$ contains a subword of the form $b^{-1}a^n b$.
		\item A path is end-essential if all good representatives do not end in a tip.
		\item A path $p\in\cP$ is swiss if a good representative $\omega$ contains a subword of the form $b^{-1}a^n ba^mb^{-1}$ or $ba^nb^{-1}a^m b$. A path that is not swiss is called nepalese.
		\item Let $\omega$ be a good representative of $p\in \cP$, an end-essential nepalese path with a tip. Let $\omega_1$ be the first part of $\omega$, let $B$ be the tip and $\omega_2$ the remainder. Then $p$ has two triplets at $B$, namely the paths associated to $\omega_1 a B a^{-1} \omega_2$ and $\omega_1 a^{-1} Ba \omega_2$.
		\item Let $\omega$ be a good representative of $p\in \cP$, an end-essential nepalese path with a valley. Let $\omega_1$ be the first part of $\omega$, let $V$ be the valley and $\omega_2$ the remainder. Then $p$ has a twin at $V$, namely the path associated to $\omega_1 a^{-1} V a b^{-1} \omega_2$.
		\item Let $\sim$ be the sibling relation defined on end-essential nepalese paths, induced by setting $p\sim q$ if they belong to the same set of triplets or the same set of twins.
		\item We denote by $G_\sim$ the graph induced by the sibling relation on end-essential nepalese paths. Two paths in the same connected component are called relatives.
		\item For a path $p\in\cP$ define the function $c:\cP\rightarrow\NN$ by setting $c(p)$ equal to the number of valleys and tips.
 	\end{enumerate}	
\end{definition}

We note that the sibling relation is not an equivalence relation, since it is not transitive and not reflexive. But since it is symmetric $\sim$ clearly defines a graph on the set of end-essential nepalese paths. This graph is not connected. We show the following.

\begin{figure}
	\centering
	\subcaptionbox{Swiss.\label{swiss}}[0.45\linewidth][l]
	{\includegraphics[width=0.45\textwidth]{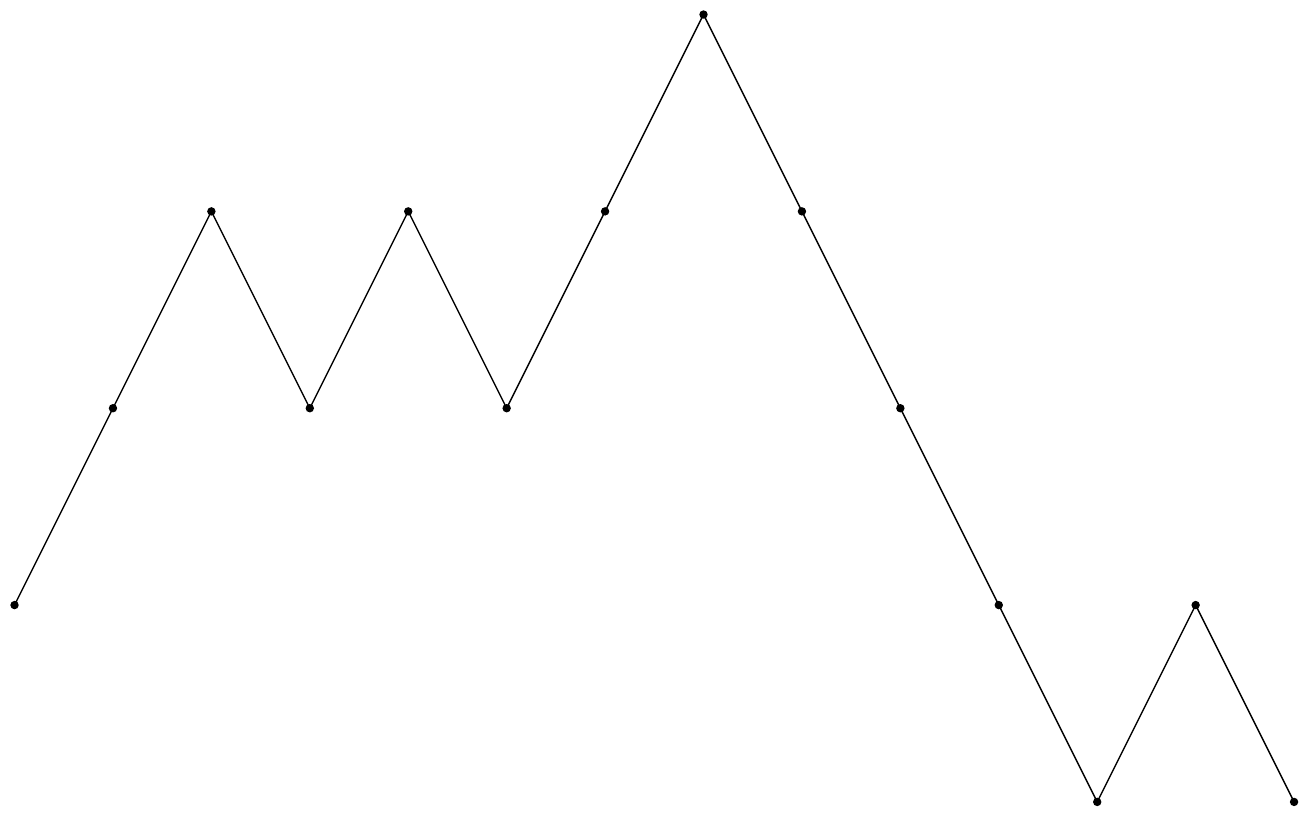}}
	\hspace{20pt}
	\subcaptionbox{Nepalese.\label{nepalese}}[0.45\linewidth][r]
	{\includegraphics[width=0.45\textwidth]{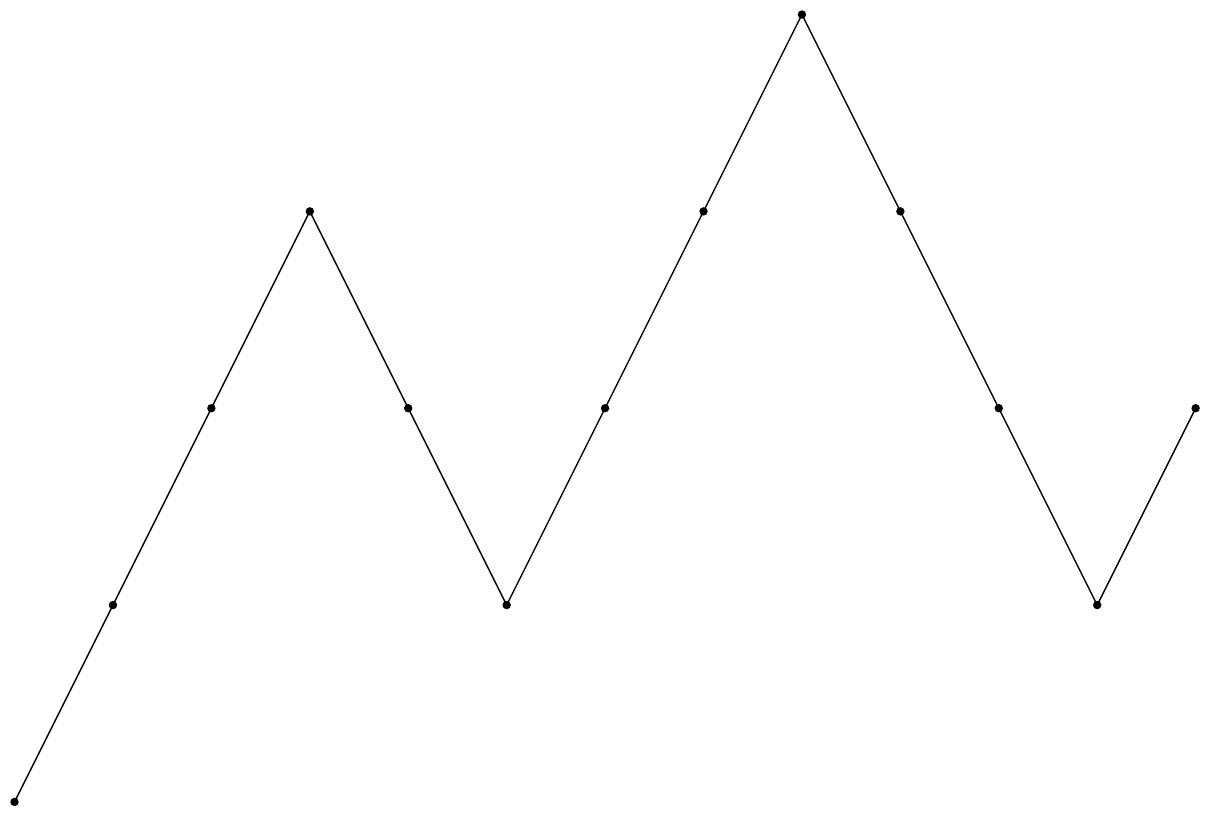}}
	\caption{An artistic interpretation of a swiss and nepalese path. The reader should think of these paths as paths in the Bass-Serre tree \ref{sideview}. Note that the swiss path has valleys and tips that share an edge, while for the nepalese path they are disjoint. Moreover, in this example, the nepalese path is end-essential while the swiss path is not.}\label{paths}
\end{figure}

\begin{theorem}
	$\Ker(\phi) \simeq \mathbb{F}_\infty$. In particular, let $\mathcal{O}$ be a set of words in $\{a,b\}^*$ such that 
	\begin{itemize}
		\item Each word represents an end-essential nepalese path and each such path is represented by at most one word.
		\item For each connected component of $G_\sim$, there is exactly one path that is represented by a word in $\mathcal{O}$.
	\end{itemize}
	Then the symmetric set
	\[ \mathcal{S}= \{ \omega a^{i} [a^b,a]^{j} a^{-i}\omega^{-1} \;\vert\; j\in \{-1,1 \},i\in\{ 0,1\}, \omega\in\mathcal{O}\} \] 
	freely generates  $\Ker(\phi)$. In particular this shows that the action of $\Ker(\phi)$ on the Bass-Serre tree is free.
\end{theorem} 
\begin{proof}
	{\em Generation.}
	We first show that $\mathcal{S}$ is generating. Let $\tilde{\mathcal{O}}$ be a set containing for each geodesic path in the Bass-Serre tree exactly one word which is a good representative. Then it can easily be seen that $\Ker(\phi)$ is generated by 
	\[\mathcal{S}= \{ \omega a^{i} [a^b,a]^{j} a^{-i}\omega^{-1} \;\vert\; j\in \{-1,1 \},i\in\{ 0,1 \}, \omega\in\tilde{\mathcal{O}}\}. \] 
	Hence we can reduce the number of paths needed. Suppose $\omega$ represents a swiss path, then it is of the form $\omega_1 bab^{-1}ab\omega_2$ or $\omega_1 b^{-1}abab^{-1}\omega_2$ (up to possibly taking other exponents of the $a$'s). We treat these two cases for $i=0$ and $j=1$, the other cases are similar. Note that for the first calculation we consider the conjugated group element and in the second we only consider the group element representing the path. As we are conjugating the paths this leads to redundant notation. We write the first in full for the readers convenience.
	\begin{align*}
		\omega [a^b,a] \omega^{-1} 	&= 	\omega_1 bab^{-1}ab\omega_2 [a^b,a] \omega_2^{-1} b^{-1}a^{-1}ba^{-1}b^{-1}\omega_2\\
									&=	\omega_1 [a^b,a] \omega_1^{-1} \omega_1 aba \omega_2 [a^b,a] \omega_2^{-1} a^{-1}b^{-1} a^{-1} \omega_1^{-1} \omega_1 [a^b,a]^{-1} \omega_1^{-1}, \text{ or }\\[10pt]
		\omega 	&= 	\omega_1 b^{-1}abab^{-1}\omega_2 \\
									&=	\omega_1 a^{-1} b^{-1} [a^b,a] ba \omega_1^{-1} \omega_1 a^{-1}b^{-1}a^4 \omega_2 .
 	\end{align*}
 	Hence we have reduced the problem to the shorter paths $\omega_1$ and $\omega_1 aba \omega_2$, and $\omega_1 a^{-1} b^{-1}$ and \linebreak $\omega_1 a^{-1}b^{-1}a^4 \omega_2$ respectively. An induction hypothesis on the length of the paths finishes the job.\newline
 	
 	Hence we may suppose our paths are nepalese. We suppose they are not end-essential. So they are represented by a word $\omega$ of the form $\omega_1 bab^{-1}$. Note that the element $bab^{-1}a^i[a^b,a]a^{-i}ba^{-1}b^{-1}$ is in the subgroup $\mathbb{F}_2$ generated by $\{ [a^b,a], a[a^b,a]a^{-1}\}$. Hence, because of its specific form, it can be written as one of the generators (or their inverses), or a product of at most four generators (or their inverses). Say $g_1$ or $g_1g_2g_3g_4$. In particular  we can rewrite $\omega a^i [a^b,a]a^{-i} \omega^{-1}$ as $\omega_1 g_1 \omega_1^{-1}$ or $\omega_1 g_1\omega_1^{-1} \omega_1 g_2 \omega_1^{-1}\omega_1 g_3 \omega_1^{-1}\omega_1 g_4 \omega_1^{-1}$, where $\omega_1$ is an end-essential path.\newline
 	
 	\begin{figure}
 		\centering
 		\subcaptionbox{Twins.\label{twinpaths}}[0.45\linewidth][l]
 		{\includegraphics[width=0.45\textwidth]{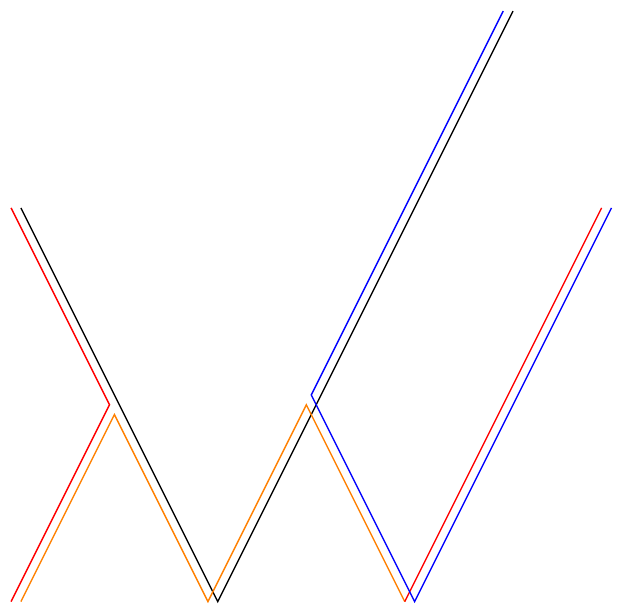}}
 		\hspace{20pt}
 		\subcaptionbox{A group element in $\Ker(\phi)$ represented as a path.\label{conjugate}}[0.45\linewidth][r]
 		{\includegraphics[width=0.45\textwidth]{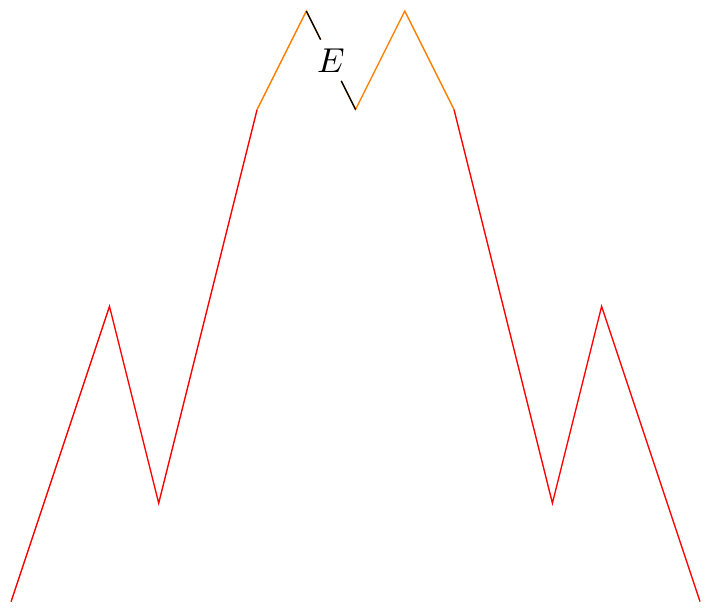}}
 		\caption{On the left: The black path (correspongding to $\omega$) has a valley $V$. Its twin at $V$ is given by the blue path. Note that the red paths corresponds to $\omega_1 a^{-1}b^{-1}$ and $ba\omega_1^{-1}$, and the orange path to the normal generator $[a^b,a]$. On the right: An element in $\Ker(\phi)$, red paths are the paths defined by $\omega$, the orange path corresponds to $[a^b,a]$, the second edge  $E$ of the commutator path seperates the Bass-Serre tree into two connected components.}\label{sexypaths}
 	\end{figure}

  	Finally consider $\omega$, a good representative for an end-essential nepalese path $p$. Let $\omega\notin \mathcal{O}$. We will show that $\omega$ can be written as a product of one (any) of its siblings and paths $q$ such that $c(q)< c(p)$. Consider two triplets $\omega_1 B \omega_2 $ and $\omega_1 a B a^{-1} \omega_2$, where $B$ is a tip. Then
	\begin{align*}
		\omega 			&= 	\omega_1 B\omega_2 \\
						&=	\omega_1 B a B^{-1} a^{-1} \omega_1^{-1} \omega_1 a B a^{-1} \omega_2 \\
						&= 	\omega_1 [a^b,a] \omega_1^{-1} \omega_1 (a B a^{-1} \omega_2) \text{, or}\\[10pt]
		\omega 			&= 	\omega_1 B\omega_2 \\
						&=	(\omega_1 a^{-1}) a B a^{-1} B^{-1} (a \omega_1^{-1}) \omega_1 a^{-1} B a \omega_2 \\
						&= 	(\omega_1 a^{-1}) [a^b,a]^{-1} (a \omega_1^{-1})( \omega_1 a^{-1} B a \omega_2).
	\end{align*} 	
 	Clearly $c(\omega_1)< c(\omega)$. On the other hand consider two twins $\omega_1  V  b^{-1} \omega_2$ and $\omega_1 a^{-1} V a b^{-1} \omega_2$. Note that the valley can either be of the form $b^{-1}ab$ or $bab^{-1}$. We suppose it is the former, the argument for the latter is similar. Note also Figure \ref{twinpaths} for a visual interpretation. One has that
 	\begin{align*}
 		\omega 			&= 	\omega_1 b^{-1} a b\omega_2 \\
 						&=	(\omega_1 a^{-1}b^{-1}) ba b^{-1}ab a^{-1} b^{-1}a^{-1} (ba\omega_1^{-1}) \omega_1 a^{-1}b^{-1} aba\omega_2\\
 						&=  (\omega_1 a^{-1}b^{-1}) [a^b,a] (ba\omega_1^{-1}) \omega_1 a^{-1}b^{-1} aba\omega_2.
 	\end{align*} 	
	The case $b^{-1}a^{-1}b$ is similar, obtaining a conjugate of $ [a^b,a]^{-1}$ in the first term. Since each connected component in $G_\sim$ is represented in $\mathcal{O}$ by some element, each end-essential nepalese path can be written as a product of the designated relative in $\mathcal{O}$ and elements with less tips or valleys. One ends the argument by induction on the number of tips and valleys. \newline
	
 	{\em Freeness.} We are now ready to prove the given basis is free. Let $\omega\in\mathcal{O}$, then $\omega a^i [a^b,a]a^{-i}\omega^{-1}$ does not necessarily represent a geodesic path (we call the path $p$), it may have backtracking when the last edge of $\omega$ and the first of $[a^b,a]$ overlap. Similarly it might also overlap on the last edge of $[a^b,a]$ and the first of $\omega^{-1}$. Let $E$ be the second edge of $[a^b,a]$ as a subpath of $p$ (see Figure \ref{conjugate}), then it seperates the tree in two conncected components (say $E^-$ and $E^+$). We claim that applying any other element $\tilde \omega a^s [a^b,a]a^{-s} \tilde\omega^{-1}$of $\mathcal{S}$, one will never traverse $E$ again. There are several possibilities.
 	\begin{itemize}
 		\item Firstly suppose $\tilde \omega = \omega$, then this reduces to the fact that $\langle [a^b,a], a[a^b,a]a^{-1} \rangle$ is free on two generators.
 		\item Consider the concatenated path $\omega a^i [a^b,a]a^{-i}\omega^{-1}\tilde \omega a^s [a^b,a]a^{-s} \tilde\omega^{-1} $. Since $\tilde \omega$ is end-essential and nepalese, this path remain in $E^+$, once having crossed $E$. Let $E_{\tilde \omega}$ be the second edge in the second commutator $[a^b,a]$. Suppose $E_{\tilde \omega}$  is not part of the path defined by $\omega^{-1}$. In particular, to cross $E$, we first need to cross $E_{\tilde \omega}$, which means we have merely displaced the problem.
 		\item So we may suppose $E_{\tilde \omega}$ is part of $\omega^{-1}$. Note that this is only possible when $\omega^{-1}$ contains a tip or a valley. In this case the composition $\omega^{-1} \tilde\omega [a^b,a]\tilde{\omega}^{-1}$ defines a sibling of $\omega^{-1}$. It is impossible, using a sequence of siblings, to cross $E$, since for every connected component of the sibling relation we only allowed one path to be contained in $\mathcal{O}$, which is $\omega$.
 	\end{itemize}
\end{proof}

\begin{remark}
	Note that if one is willing to give up the generating set, then there is a quick and clean argument to show that $\Ker(\phi)$ acts freely on the Bass-Serre tree. The stabiliser of each vertex is given by a conjugate of $\langle a \rangle$. So to show that $\Ker(\phi)$  acts freely, it is enough to show that the intersection with each of the stabilisers is trivial. Given that $\Ker(\phi)$ is normal, it suffices to show that its intersection with $\langle a \rangle$ is trivial. Since $a^n$ is mapped to $a^{2n}$ this is true.
\end{remark}

Furthermore we look at the limit case of applying the morphism $\phi$.
\begin{proposition}
$\Ker(\phi^n)$ is normally generated by $\{ [b^m a b^{-m}, a] \;\lvert\; 0<m\le n    \}$
\end{proposition}
\begin{proof}
By induction. 
\end{proof}

\begin{corollary}\label{limit}
$\BS / \bigcup_n \Ker(\phi^n) = \langle a,b \;\lvert \; [b^m a b^{-m}, a] \text{ where } m\in \mathbb{N} \rangle.$
\end{corollary}

Call this group $L$. Note that $\forall n\in\mathbb{N}$ there is an element $\alpha\in L$, such that $\alpha^{2^n}= \overline{a}$. I.e $\overline{a}$ has infinitely many roots. More specifically the presentation induced by $\alpha$ and $\overline{b}$ is the same as the one given in Corollary \ref{limit}. One sees that the group generated by $\langle \overline{b^m a b^{-m}} \;\vert\; m\in\mathbb{N} \rangle $ is isomorphic to the dyadic rationals $\mathbb{Z}[1/2]$.

\section{Another interpretation of non-Hopfianness and a continuous map}\label{levittpres}

We present another interpretation on the non-Hopfianness of $\BS$, which was communicated to the author by Gilbert Levitt. In some sense, this is a more topological version of Remark 6.10 in \cite{levitt}. The paper treats generalised Baumslag-Solitar groups, which are fundamental groups of graphs of groups. This is also the setting of Remark 6.10. These groups can be  retrieved as the fundamental groups of certain 2-dimensional simplicial complexes (see \cite[Section 7]{delgado}). We will not introduce graphs of groups and associated concepts, but will rather immediately take the topological point of view adapted to $\BS$. As a last remark note that the forthcoming example is not particular to the case of $\BS$, but can easily be altered to apply to other Baumslag-Solitar groups. The same is true for Figure \ref{Levittpic}. \newline 

First as small detour. We note the following fact: in $\BS$, the elements $a^2$ and $a^3$ are ``the same" since they are conjugates.  What we mean is that they behave in the same way, which is made precise in the following lemma.
\begin{lemma}
	The group $\BS$ has an abstract group presentation, where the generators can be interpreted as either $\{a^2,b\}$ or $\{a^3,b\}$. This presentation is given by $\langle \lambda, \mu \;\vert\; \lambda^2 = \mu \lambda \mu^{-1}\lambda^{-1}\mu\lambda \mu^{-1}\rangle$.
\end{lemma}
\begin{proof}
	Starting from the usual presentation $\langle a,b \;\vert \; ba^2b^{-1} =  a^3\rangle$, one can apply two sequences of Tietze transformations to obtain the presentation above. We present the calculations for $a^2$: 
	\begin{align*}
	\langle a,b \;\vert \; ba^2b^{-1} =  a^3\rangle &= \langle a, \lambda, b \;\vert \; \lambda= a^2,  \;  ba^2b^{-1} =  a^3\rangle\\
	&= \langle a, \lambda, b \;\vert \; \lambda= a^2,  \;  ba^2b^{-1} =  a^3,  \;  a= b\lambda b^{-1} \lambda^{-1}\rangle\\
	&= \langle \lambda, b \;\vert \; \lambda= b\lambda b^{-1} \lambda^{-1}b\lambda b^{-1} \lambda^{-1},  \;  b\lambda b^{-1} =  b\lambda b^{-1} \lambda^{-1}\lambda\rangle\\
	&= \langle \lambda, b \;\vert \; \lambda^2= b\lambda b^{-1} \lambda^{-1}b\lambda b^{-1}\rangle\\
	&= \langle \lambda, \mu\;\vert \; \lambda^2= \mu\lambda \mu^{-1} \lambda^{-1}\mu\lambda \mu^{-1}\rangle.
	\end{align*}
	We can obtain the same presentation for $a^3$ by doing similar calculations.
\end{proof}
We define $\phi'$ to be the morphism 
\[    BS(2,3) \rightarrow BS(2,3) :  
\begin{cases}
a \mapsto a^3\\
b\mapsto b \\
\end{cases}
. \]
From a group presentation point of view the following is true: since $a^2$ and $a^3$ behave in a completely similar way the morphisms $\phi$ and $\phi'$ are in fact also the same morphism! As a consequence they have the same kernel.\newline

\begin{figure}[b]
	\centering
	\includegraphics[width=0.5\textwidth]{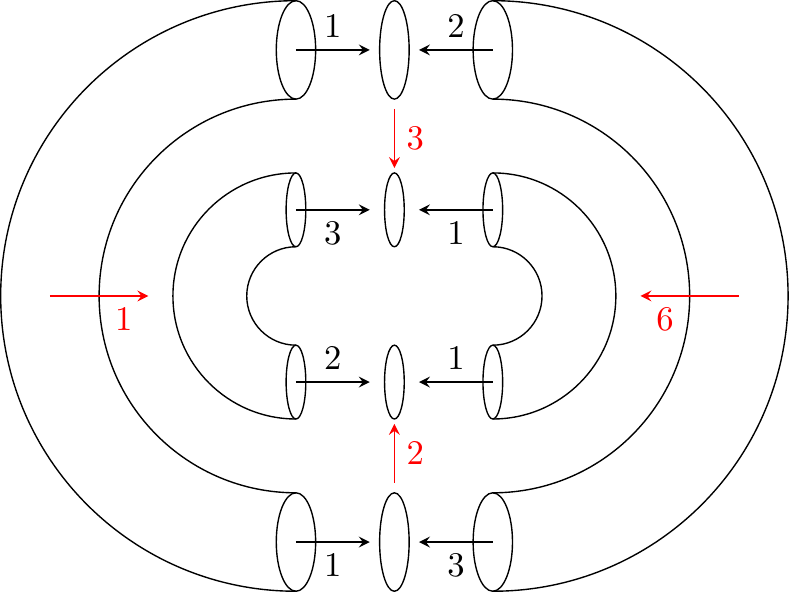}
	\caption{$\BS$ as a fundamental group and $\phi$ as a morphism induced by a continuous self-map (in red) of the space $B$.}\label{Levittpic}
\end{figure}

We start by describing a space $B$ such that $\BS$ is its fundamental group. The basic building blocks are two circles $S_0$ and $S_1$, and two annuli $A_0:= S_a^0 \times [0,1]$ and $A_1:= S_a^1 \times [1,2]$. Now we glue the blocks together. For $A_0$ we attach $S_a^0\times 0$ to $S_0$ along the identity map $e^{i\theta}\mapsto e^{i\theta}$. We do the same for $S_a^0\times 1$ to $S_1$. For $A_1$ we attach  $S_a^1\times 1$ to $S_0$ along the map $e^{i\theta} \rightarrow e^{3i\theta}$, and $S_a^1\times 2$ to $S_1$ along the map $e^{i\theta} \rightarrow e^{2i\theta}$. This space is depicted two times in Figure \ref{Levittpic}, being once given by the outer annuli and circles, and a second time by the inner annuli and circles. The reader can easily verify that the fundamental group of $B$ at $(1,\frac{1}{2})$ is generated by two loops
\begin{align*}
&\gamma_1: [0,1]\rightarrow B: \theta\mapsto \left(e^{2\pi i \theta}, \frac{1}{2}\right)\\
&\arraycolsep=1.4pt\def\arraystretch{1.5} 
\gamma_2: [0,1]\rightarrow B: x \mapsto \left\{\begin{array}{ll} \left(1,\frac{1}{2}-2x\right) & \text{ if } x\in  \left[0,\frac{1}{4}\right] \\
\left(1,1+2x-\frac{1}{2}\right) & \text{ if } x\in \left[ \frac{1}{4},\frac{3}{4} \right] \\
\left(1,\frac{5}{2}-2x\right) & \text{ if } x\in \left[\frac{3}{4},1\right]
\end{array}\right. , 
\end{align*}

where $\pi_1(B, (1,\frac{1}{2}))$ is isomorphic to $\BS$ by mapping $a$ to $[\gamma_1]$ and $b$ to $[\gamma_2]$. \newline

Finally the map $\tilde \phi: B\rightarrow B$ is given in red in Figure \ref{Levittpic}. We define the map piecewise
\begin{itemize}
	\item $A_0\rightarrow A_1: (e^{i\theta},x)\mapsto (e^{i\theta},2-x)$
	\item $A_1\rightarrow A_0: (e^{i\theta},x)\mapsto (e^{6i\theta},2-x)$
	\item $S_0\rightarrow S_1: e^{i\theta} \mapsto e^{2i\theta}$ 
	\item $S_1\rightarrow S_0: e^{i\theta} \mapsto e^{3i\theta}$ .
\end{itemize}

It is easy to verify that this map respects the gluing of the building blocks, hence is continuous on $B$. In particular $\tilde \phi$ induces a map on the fundamental group of $B$, being $\tilde{\phi}_*: \pi_1(B, (1,\frac{1}{2}))\rightarrow \pi_1(B, (1,\frac{3}{2}))$. One could check that $\Ker(\tilde{\phi}_*)$ is exactly $\Ker(\phi)$ (and thus also $\Ker(\phi')$). So from a group presentation point of view, we can say that $\tilde \phi$ induces the group morphism $\phi$. However, suppose we really want to recover $\phi$, then we would need a map from $\pi_1(B, (1,\frac{1}{2}))$ to itself. Consider the path 
\[\arraycolsep=1.4pt\def\arraystretch{1.5} 
 p^1: [0,1]\rightarrow B : x \mapsto \left\{\begin{array}{ll}  \left(1, \frac{1}{2} + x\right)   & \text{ if } x\in 							\left[0,\frac{1}{2}\right] \\
\left(1, 2-x\right)   & \text{ if } x\in \left[\frac{1}{2},1\right]
\end{array}\right..  \]
This path imposes an isomorphism $p^1_*: \pi_1(B, (1,\frac{3}{2})) \rightarrow \pi_1(B, (1,\frac{1}{2})) $. Now composing, we obtain that $p^1_*\circ \tilde{\phi}_*$ is exactly $\phi$. If, on the other hand, we want to recover $\phi'$, we only need to compose  $\tilde{\phi}_*$ with $p^2_*$, where $p^2$ is the path given by
\[\arraycolsep=1.4pt\def\arraystretch{1.5} 
 p^2: [0,1]\rightarrow B : x \mapsto \left\{\begin{array}{ll}  \left(1, \frac{1}{2} - x\right)   & \text{ if } x\in 							\left[0,\frac{1}{2}\right] \\
\left(1, \frac{1}{2} + x\right)   & \text{ if } x\in \left[\frac{1}{2},1\right]
\end{array}\right. .  \]

\bibliographystyle{alpha}
\bibliography{biblio}

\end{document}